\documentclass[a4paper,12pt]{article}
\usepackage[utf8]{inputenc}
\title{Error estimates in balanced norms of finite element methods
       for higher order reaction-diffusion problems}
\author{Sebastian Franz\footnote{
          Institute of Scientific Computing, Technische Universit\"at Dresden, Germany.
          \mbox{e-mail}: sebastian.franz@tu-dresden.de}
        \and
        Hans-G. Roos\footnote{
          Institute of Numerical Mathematics, Technische Universit\"at Dresden, Germany.
          \mbox{e-mail}: hans-goerg.roos@tu-dresden.de}        
        }
\date{\today}

\usepackage{fancyhdr} 
\usepackage{pgfpages}

\usepackage{amsmath}
\usepackage{amssymb}
\usepackage{amsthm}

\makeatletter
\let\my@saved@original@eqref\eqref 
\renewcommand*{\eqref}[1]{
  \begingroup
    \let\normalfont\relax
    \my@saved@original@eqref{#1}
  \endgroup
}
\makeatother

\usepackage{cite}

\usepackage{booktabs}

\usepackage{color}
\newcommand{\ord}[1]{\mathcal{O}\left(#1\right)}
\DeclareMathOperator{\meas}{meas}
\newcommand{\e}{\mathrm{e}}

\newcommand{\grad}{\nabla}

\newcommand{\eps}{\varepsilon}
\newcommand{\norm}[2]{\|{#1}\|_{#2}}
\newcommand{\snorm}[2]{|{#1}|_{#2}}

\newcommand{\tnorm}[1]{\left|\!\!\;\left|\!\!\;\left| {#1}
                       \right|\!\!\;\right|\!\!\;\right|}
\newcommand{\enorm}[1]{\tnorm{#1}_\eps}

\newcommand{\I}{\mathcal{I}}

\newcommand{\QS}{\mathcal{Q}}

\newcounter{tmp}
\newcommand{\makeballnumber}[1]{\setcounter{tmp}{\theenumi}%
\setcounter{enumi}{#1}%
\leavevmode \csname beamer@@tmpl@enumerate item\endcsname%
\setcounter{enumi}{\thetmp}}
\newcommand{\makeball}{\leavevmode \csname beamer@@tmpl@itemize item\endcsname}

\definecolor{seb}{rgb}{0.9,0,0}

\makeatletter
\renewcommand*\env@matrix[1][r]{\hskip -\arraycolsep
  \let\@ifnextchar\new@ifnextchar
  \array{*\c@MaxMatrixCols #1}}
\makeatother
\numberwithin{equation}{section}

\renewcommand{\phi}{\varphi}
\DeclareMathAlphabet{\mathcal}{OMS}{cmsy}{m}{n}
\newcommand{\bnorm}[1]{\tnorm{#1}_{b}}

\theoremstyle{plain}
\newtheorem{theorem}{Theorem}[section]
\newtheorem{lemma}[theorem]{Lemma}
\newtheorem{assumption}[theorem]{Assumption}
\newtheorem{remark}[theorem]{Remark}

\begin{document}
 \pagestyle{fancy}
  \maketitle

  \begin{abstract} 
    Error estimates of finite element methods for reaction-diffusion
    problems are often realised in the related energy norm. In the singularly perturbed
    case, however, this norm is not adequate. A different scaling of the $H^m$ seminorm
    for $2m$-th order problems leads to a balanced norm which reflects the layer behaviour 
    correctly.
    
    We prove error estimates in such balanced norms and improve thereby existing
    estimates known in literature.
  \end{abstract}



\section{Introduction}
  We shall examine the finite element method for the numerical
  solution of a singularly perturbed linear elliptic $2m-$th order
  boundary value problem in two dimensions. In the weak form it is given by
  \begin{equation}\label{1.1}
    \eps^{2k}(\grad^m u,\grad^m v)+\tilde a(u,v)=(f,v)\quad
    \forall v\in H_0^m(\Omega),
  \end{equation}
  where $\Omega=(0,1)^2$, $0<\eps \ll 1$ is a small positive parameter, $1\le k\le m$ and $f$ is sufficiently smooth.
  We assume that the bilinear form $\tilde a(\cdot,\cdot)$ is related to a $2(m-k)-$th order operator
  and $\tilde a(u,u)$ is equivalent to $\norm{u}{H^{m-k}}^2$.

  The Lax-Milgram theorem tells us that the problem has a unique solution $u\in H_0^m(\Omega)$
  which is sufficiently smooth for smooth data and satisfies in the energy norm
  \begin{equation}\label{1.2}
    \enorm{u}:=\eps^{k}|u|_{H^m}+\norm{u}{H^{m-k}}\lesssim \norm{f}{L^2}.
  \end{equation}
  Here and in the following we use the following notation: if $A\lesssim B$ then there exists a (generic)
  constant $C$ independent of $\eps$ (and later also of the mesh used) such that $A\le C\,B$.

  The error of a finite element approximation $u^N\in V^N$ satisfies
  \begin{equation}\label{1.3}
    \enorm{u-u^N}\lesssim \min_{v^N\in V^N}\enorm{u-v^N}
  \end{equation}
  for any finite dimensional space $V^N\subset H_0^m(\Omega)$.

  If we use $C^{m-1}$-splines, piecewise polynomial of degree $2m-1$, on a properly defined Shishkin mesh 
  with $N$ cells in each direction, then one can prove for the interpolation error of the Hermite
  interpolant $u^I\in V^N$
  \begin{equation}\label{1.4}
    \enorm{u-u^I}\lesssim \left(\eps^{1/2}(N^{-1}\ln N)^m+N^{-(m+1)}\right).
  \end{equation}
  It follows that the error $u-u^N$ also satisfies such an estimate. Some special one-dimensional cases are
  discussed, for instance, in \cite{LX06,SS95,Xenophontos17}.

  However, a typical boundary layer function $\eps^{m-k}\exp(-x/\eps)$ of our given problem
  measured in the norm $\enorm{\cdot}$ is of order $\ord{\eps^{1/2}}$.
  Consequently, error estimates in this norm are less valuable as for convection
  diffusion equations.
  Therefore, we ask the fundamental question:\\
  \emph{Is it possible to prove error estimates in the balanced norm}
  \begin{equation}\label{1.5}
    \bnorm{v}:=\eps^{k-1/2}|v|_{H^m}+\norm{v}{H^{m-k}}\quad ?
  \end{equation}

  For higher order equations ($m\ge 2$), even in 1d nothing is known concerning estimates in the
  balanced norm for the Galerkin finite element method. The only exception is \cite{FrR15}, where a fourth-order problem
  is discretised with a mixed finite element method.
  
  The outline of this paper is as follows. In Section~\ref{sec:2} we present a new idea to 
  derive balanced error estimates for second order problems, improving the result in \cite{RSch15}.
  In Section~\ref{sec:3} we generalise the idea from Section~\ref{sec:2} to higher order problems
  in detail for the 1d case and give guiding principles for the (very technical) 2d case.
  
  \textbf{Notation:} We denote by $(\cdot,\cdot)_D$ the $L^2$-scalar product on $D$ and by $\norm{\cdot}{L^2(D)}$
  the associated $L^2$-norm over $D$. Furthermore by $\snorm{\cdot}{H^k(D)}$, $\norm{\cdot}{H^k(D)}$ and $\norm{\cdot}{W^{k,\infty}(D)}$
  we denote the Sobolev-seminorm and norms in $H^k(D)=W^{k,2}(D)$ and $W^{k,\infty}(D)$. 
  In the case of $D=\Omega$ we may skip the reference to the domain.

\section{An improved estimate in a balanced norm for second order problems}\label{sec:2}

  Let us consider the case $m=k=1$ and the discretization of
  \begin{gather}\label{eq:standard1}
    \eps^2  (\grad u,\grad v) + (c u,v) = (f,v)\quad \forall v \in V=H_0^1(\Omega),
  \end{gather}
  where $c\geq\gamma>0$ by linear finite elements on S-type meshes \cite{RL99}. 
  In \cite{RSch15} it was proved (on a Shishkin mesh)
  \begin{equation}\label{2.5}
    \bnorm{u-u^N}\lesssim N^{-1}(\ln N)^{3/2}+N^{-2}.
  \end{equation}

  It was an open question to remove the factor $(\ln N)^{1/2}$ from \eqref{eq:standard1}. Here
  we modify the technique from \cite{RSch15} to realise that goal and use the same technique in 
  Section \ref{sec:3} for higher order problems.

  In \cite{RSch15} the $L^2$-projection $\pi u\in V^N$ from $u$ was used 
  instead of the Lagrange interpolant. Based on
  \[
    u-u^N=u-\pi u+\pi u-u^N
  \]
  we estimated for constant $c$ the discrete error $\pi u-u^N$ starting from:
  \begin{multline}\label{st}
    \enorm{\pi u-u^N}^2
      \lesssim \eps^2\norm{\grad (\pi u-u^N)}{L^2}^2+c\,\norm{\pi u-u^N}{L^2}^2\\
      =\eps^2(\grad(\pi u-u),\grad (\pi u-u^N))+c\,(\pi u-u,\pi u-u^N).
  \end{multline}
  With $(\pi u-u,\xi)=0$ for $\xi\in V^N$, the last term vanishes and the problem was to estimate $\norm{\grad(\pi u-u)}{L^2}$.
  The use of the global projection leads to difficulties, especially in 2D: it is known that the
  $L^2$ projection is not on every mesh $L^p$ stable, and there are examples which show that for the
  $W^{1,p}$ stability restrictions on the mesh are necessary even in the one-dimensional case
  \cite{CT87,Oswald13}.
  
  Here we modify the definition of the projection into $V^N$, the space of piecewise polynomials of degree $p\geq 1$ in each coordinate direction. 
  In order to do so we start by defining our mesh for the number $N$ of cells in each direction divisible by 4. 
  Let $\phi$ be a monotonically increasing function with $\phi(0)=0$, $\phi(1/2)=\ln N$ -- 
  the so-called \textit{mesh-generating function} -- and $\psi:=\ln(-\phi)$ the \textit{mesh characterising
  function}, see \cite{RL99}. Furthermore let $\lambda:=\sigma\eps\ln N$ be the transition parameter,
  where $\sigma$ is a user chosen parameter to be specified later and $\lambda\leq 1/4$ is assumed.
  
  The idea for defining the transition parameter comes is related to the Assumption~\ref{ass:decomp} on a 
  solution decomposition, see \cite{HK90}.
  
  \begin{assumption}\label{ass:decomp}
    We assume the decomposition $u=v+\sum\limits_{k=1}^4w_k+\sum\limits_{k=1}^4c_k$ into a smooth part $v$, 
    boundary layer parts $w_k$ and corner layer parts $c_k$. To be more precise we assume
    for $0\leq i,j\leq p+1$
    \begin{align*}
      |\partial_x^i\partial_y^j v(x,y)| &\lesssim 1,\\
      |\partial_x^i\partial_y^j w_1(x,y)| &\lesssim\eps^{-i}\exp(-x/\eps),\\
      |\partial_x^i\partial_y^j c_1(x,y)| &\lesssim\eps^{-(i+j)}\exp(-(x+y)/\eps)
    \end{align*}
    and similarly for the remaining terms.
  \end{assumption}
  
  Now we have $|w_1(\lambda,y)|\lesssim N^{-\sigma}$ and the size of the layer components in $\Omega_c$
  can be adjusted by $\sigma$. 
  
  The mesh-points are then defined by
  \[
    x_i=y_i=\begin{cases}
              \sigma\eps\phi\left(\frac{2i}{N}\right),& i\in\{0,\dots,N/4\},\\
              \lambda+\left(\frac{4i}{N}-1\right)\left(\frac{1}{2}-\lambda\right),& i\in\{N/4,\dots,3N/4\},\\
              1-\sigma\eps\phi\left(2-\frac{2i}{N}\right),& i\in\{3N/4,\dots,N\}.
            \end{cases}
  \]
  By drawing axis-parallel lines through the so-defined mesh points we obtain an S-Type mesh with 
  equidistant cells in the coarse region $\Omega_{c}:=(\lambda,1-\lambda)^2$ and anisotropic cells in the 
  layer region $\Omega\setminus\Omega_{c}$. Note that in the layer region the small mesh-sizes
  can be estimated by $h_i:=x_{i-1}-x_i\leq h$ and $k_j=y_{j+1}-y_j\leq h$ with
  \begin{gather}\label{eq:hbound1}
    \eps N^{-1}\ln N\lesssim h\lesssim \eps,
  \end{gather}
  and similarly for the $y$-direction.
  
  \begin{assumption}\label{ass:convex}
    Let the mesh-generating function $\phi$ be convex.
  \end{assumption}
  
  Most of the generating functions of S-type-meshes fulfil this assumption, i.e. the most prominent two
  \begin{itemize}
    \item Shishkin mesh: $\phi(t)=2t\ln N$,
    \item Bakhvalov-S-mesh: $\phi(t)=-\ln(1-2t(1-N^{-1}))$.
  \end{itemize}
  As a result of Assumption~\ref{ass:convex} the cells in the layer region adjacent to the transition line
  have a width of $h$ orthogonal to the transition line. We then define another domain by enlarging $\Omega_c$
  one ply of cells in each direction:
  \[
    \Omega_c^*:=(\lambda-h,1-(\lambda-h))^2.
  \]
  Let us denote by $\I$ the piecewise Gauß-Lobatto interpolation operator that uses as local interpolation points
  the quadrature nodes $(\hat x_k,\,\hat k_\ell)$ for $k,\ell\in\{0,\dots,p+1\}$ of the Gauss-Lobatto quadrature 
  rule. Furthermore, we denote by $\pi$ the weighted, $\Omega_c$-global $L^2$-projection $\pi v\in V^N$ defined by
  \[
    (c(v-\pi v),\omega)_{\Omega_c}=0\quad\forall \omega\in V^N,
  \]
  where we have denoted by $(\cdot,\cdot)_{\Omega_c}$ the restriction of the $L^2$-scalar product to $\Omega_c$.
  Additionally, we denote by $\chi_\tau\in V^N$ on each element $\tau\in\Omega_c^*\setminus\Omega_c$ the discrete function with
  \[
    \chi_\tau(\hat x_k,\hat y_\ell)=\begin{cases}
                                      1,&(\hat x_k,\hat y_\ell)\in\partial\Omega_c,\\
                                      0,&otherwise.
                                    \end{cases}
  \]
  Note that on $\Omega_c^*\setminus\Omega_c$ only two types of $\chi_\tau$ exist: They are one in either exactly one corner or on 
  exactly one side of $\tau$.
  
  Now we can finally define our new interpolation operator. Let the interpolation operator $P$ into $V^N$ for $u=v+w$, 
  where $w=\sum\limits_{k=1}^4w_k+\sum\limits_{k=1}^4c_k$, be defined by
  \begin{align*}
    Pw|_\tau&:=\begin{cases}
                  0, & \tau\subset\Omega_c,\\
                  \I w, & \tau\subset\Omega\setminus\Omega_c^*,\\
                  \I[(1-\chi_\tau)w], & \tau\subset\Omega_c^*\setminus\Omega_c,
               \end{cases}&
    Pv|_\tau&:=\begin{cases}
                  \pi v|_\tau, & \tau\subset\Omega_c,\\
                  \I v, & \tau\subset\Omega\setminus\Omega_c^*,\\
                  \I[(1-\chi_\tau) v+\chi_\tau\pi v], & \tau\subset\Omega_c^*\setminus\Omega_c.
               \end{cases}
  \end{align*}
  
  \begin{lemma}\label{lem:L2-projection}
    For any $v\in W^{p+1,\infty}(\Omega_c)$ holds
    \[
      \norm{Iv-\pi v}{L^\infty(\partial\Omega_c)}\lesssim N^{-(p+1)}.
    \]
  \end{lemma}
  \begin{proof}
    Using 
    $\pi (I v)= I v$ due to $\pi$ being a projection we have
    \begin{align*}
      \norm{Iv-\pi v}{L^\infty(\partial\Omega_c)}
        &\leq\norm{\pi(Iv-v)}{L^\infty(\Omega_c)}
        \lesssim\norm{Iv-v}{L^\infty(\Omega_c)},
    \end{align*}
    where we have used in the last step the $L^\infty$-stability of the $L^2$-projection on $\Omega_c$, see \cite{Oswald13}. 
    The result follows by standard interpolation error estimation on equidistant meshes.
    Alternatively to the $L^\infty$-stability an $L^\infty$-error estimate of the $L^2$-projection, see \cite{Nitsche75,Scott76}, could be used.
  \end{proof}
  
  We will use in the following the splitting of the error into the interpolation and discrete error given by
  \[
    u-u_N=(u-Pu)+(Pu-u^N)=:\eta+\xi.
  \]
  
  \begin{lemma}\label{lem:reaction}
    Let $\sigma\geq p+1$. Under the Assumption~\ref{ass:decomp} we have
    \[
      |(c\eta,\xi)|\lesssim\eps^{1/2}\Big(N^{-(p+1)}(\ln N)^{1/2}+\left(h+N^{-1}\max|\psi'|\right)^{p+1}\Big)\enorm{\xi}.
    \]
  \end{lemma}
  \begin{proof}
    We will prove the estimate in the coarse and remaining region separately. 
    Let us start on $\Omega_c$. 
    By definition of $P$ and the $L^2$-orthogonality of the $L^2$-error we have
    \[
      |(c\eta,\xi)_{\Omega_c}|
        =|(cw,\xi)_{\Omega_c}|
        \lesssim\norm{w}{L^2(\Omega_c)}\norm{\xi}{L^2(\Omega_c)}
        \lesssim\eps^{1/2}N^{-\sigma}\enorm{\xi}.
    \]
    In the remaining domain we have
    \[
      (c\eta,\xi)_{\Omega\setminus\Omega_c}
        =(c(u-Iu),\xi)_{\Omega\setminus\Omega_c}+(c(Iu-Pu),\xi)_{\Omega_c^*\setminus\Omega_c},
    \]
    where we extended the application of $I$ into the ply of elements around $\Omega_c$.
    For the first term it holds with a Hölder inequality
    \begin{align*}
      |(c(u-Iu),\xi)_{\Omega\setminus\Omega_c}|
        &\lesssim\left(\meas^{1/2}(\Omega\setminus\Omega_c)\norm{v-Iv}{L^\infty(\Omega\setminus\Omega_c)}
                +\norm{w-Iw}{L^2(\Omega\setminus\Omega_c)}\right)\enorm{\xi}\\
        &\lesssim\eps^{1/2}\left(N^{-(p+1)}(\ln N)^{1/2}+\left(h+N^{-1}\max|\psi'|\right)^{p+1}\right)\enorm{\xi},
    \end{align*}
    while for the second term we have using the special function $\chi\in V^N$
    \begin{align*}
      |(c(Iu-Pu),\xi)_{\Omega_c^*\setminus\Omega_c}|
        &\lesssim(\norm{Iv-Pv}{L^2(\Omega_c^*\setminus\Omega_c)}+\norm{Iw-Pw}{L^2(\Omega_c^*\setminus\Omega_c)})\enorm{\xi}\\
        &\lesssim\left(\norm{Iv-\pi v}{L^\infty(\partial\Omega_c)}+\norm{Iw}{L^\infty(\partial\Omega_c)}
               \right)\norm{\chi}{L^2(\Omega_c^*\setminus\Omega_c)}\enorm{\xi}.
    \end{align*}
    Applying Lemma~\ref{lem:L2-projection}, the boundedness of Gauss-Lobatto-basis functions and the $L^\infty$-stability of $I$
    we obtain
    \begin{align*}
      |(c(Iu-Pu),\xi)_{\Omega_c^*\setminus\Omega_c}|
        &\lesssim\meas^{1/2}(\Omega_c^*\setminus\Omega_c)\left(\norm{Iv-\pi v}{L^\infty(\partial\Omega_c)}+\norm{w}{L^\infty(\partial\Omega_c)}
               \right)\enorm{\xi}\\
        &\lesssim\eps^{1/2}\left(N^{-(p+1)}+N^{-\sigma}\right)\enorm{\xi},
    \end{align*}
    where $\meas(\Omega_c^*\setminus\Omega_c)\lesssim h\lesssim\eps$ was used. With $\sigma\geq p+1$ the proof is finished.
  \end{proof}
  
  The final ingredient for our proof is the estimation of the interpolation error in the balanced norm.
  
  \begin{lemma}\label{lem:interpol}
    Let $\sigma\geq p+1$. Under the Assumptions~\ref{ass:decomp} and \ref{ass:convex} we have
    \[
      \bnorm{\eta}\lesssim\left(h+N^{-1}\max|\psi'|\right)^p.
    \]
  \end{lemma}
  \begin{proof}
    We start by splitting the error into
    \[
      \bnorm{\eta}
        \lesssim \tnorm{\eta}_{b,\Omega_c}+\tnorm{u-Iu}_{b,\Omega\setminus\Omega_c}+\tnorm{Iu-Pu}_{b,\Omega_c^*\setminus\Omega_c}.
    \]
    By standard anisotropic interpolation error estimation we obtain
    \[
      \tnorm{u-Iu}_{b,\Omega\setminus\Omega_c}
        \lesssim \left(h+N^{-1}\max|\psi'|\right)^p.
    \]
    Using the definition of $P$ on $\Omega_c$ we have
    \begin{align*}
      \tnorm{\eta}_{b,\Omega_c}^2
        &\leq \eps\norm{\grad(v-\pi v)}{L^2(\Omega_c)}^2+
              \eps\norm{\grad w}{L^2(\Omega_c)}^2+
              \gamma\norm{v-\pi v}{L^2(\Omega_c)}^2+
              \gamma\norm{w}{L^2(\Omega_c)}^2\\
        &\lesssim \eps N^{-2p}+N^{-2\sigma}+N^{-2(p+1)}.
    \end{align*}
    For the remaining term we apply an inverse inequality. By Assumption~\ref{ass:convex}
    the small size of the cells in $\Omega_c^*\setminus\Omega_c$ is $h$ and this can be bounded from below by
    \begin{gather}\label{eq:hbound2}
      h\geq 4\sigma\eps N^{-1}\ln N,
    \end{gather}
    see also \eqref{eq:hbound1}.
    Thus we get
    \begin{align*}
      \tnorm{Iu-Pu}_{b,\Omega_c^*\setminus\Omega_c}
        &\lesssim \eps^{1/2}\norm{\grad(Iu-Pu)}{L^2(\Omega_c^*\setminus\Omega_c)}+\norm{Iu-Pu}{L^2(\Omega_c^*\setminus\Omega_c)}\\
        &\lesssim \left(\frac{\eps}{\min\{h,N^{-1}\}}+\eps^{1/2}\right)\norm{Iu-Pu}{L^\infty(\Omega_c^*\setminus\Omega_c)}\\
        &\lesssim N (N^{-\sigma}+N^{-(p+1)}),
    \end{align*}
    where Lemma~\ref{lem:L2-projection} was used in the last step. Together with $\sigma\geq p+1$ the proof is complete.
  \end{proof}
  Using these Lemmas we obtain the main result for this section.
  \begin{theorem}\label{thm:reactdiff}
    Let $\sigma\geq p+1$ and Assumptions~\ref{ass:decomp} and \ref{ass:convex} hold. 
    Then we have for the solutions $u$ of \eqref{eq:standard1} and $u^N$ of the corresponding Galerkin method
    \[
      \bnorm{u-u^N}\lesssim \left(h+N^{-1}\max|\psi'|\right)^p.
    \]
  \end{theorem}
  \begin{proof}
    Let us start with the discrete error $\xi$. Using coercivity in the energy norm and Galerkin orthogonality we have
    \[
      \enorm{\xi}^2\leq \eps^{1/2}\bnorm{\eta}\enorm{\xi}+|(c\eta,\xi)|.
    \]
    With Lemma~\ref{lem:reaction} we get
    \[
      \enorm{\xi}^2\lesssim \eps^{1/2}(\bnorm{\eta}+\left(h+N^{-1}\max|\psi'|\right)^p)\enorm{\xi}
    \]
    and therefore
    \[
      \eps^{1/2}\norm{\grad\xi}{L^2}
        \leq \eps^{-1/2}\enorm{\xi}
        \lesssim \bnorm{\eta}+\left(h+N^{-1}\max|\psi'|\right)^p.
    \]
    Together with the energy-norm result for $\xi$
    \[
      \norm{\xi}{L^2}\leq\enorm{\xi}\lesssim \left(h+N^{-1}\max|\psi'|\right)^p
    \]
    we have
    \[
      \bnorm{\xi}\lesssim \left(h+N^{-1}\max|\psi'|\right)^p.
    \]
    Now the triangle inequality and Lemma~\ref{lem:interpol} yield the assertion
    \[
      \bnorm{u-u^N}
        \leq \bnorm{\eta}+\bnorm{\xi}
        \lesssim \left(h+N^{-1}\max|\psi'|\right)^p.\qedhere
    \]
  \end{proof}
  \begin{remark}
    In \cite{Roos17} we proved for linear elements on S-type meshes the estimate
    \begin{equation}\label{S-type}
      \|u-u^N\|_{b}\lesssim h+N^{-1}(\ln N)^{1/2}\max |\psi'|
    \end{equation}
    under the assumption
    \begin{equation}
      N^{-1}\lesssim \phi(1/N).
    \end{equation}
    This assumption guarantees that the minimal mesh size ($\phi$ is convex and monotonically increasing) is not
    too small, which is guaranteed for Shishkin and Bakhvalov-Shishkin meshes, but not, for instance,
    for polynomial Shishkin-meshes.
    Our new approach improves upon the estimate \eqref{S-type} by the factor $(\ln N)^{1/2}$
    without this assumption.
  \end{remark}

\section{Higher order problems}\label{sec:3}
  Let us consider the higher-order version of our problem in 1d, i.e.
  \begin{equation}\label{3.1}
    \eps^{2k}(u^{(m)},v^{(m)})+\tilde a(u,v)=(f,v)\quad
    \forall v\in H_0^m((0,1)),
  \end{equation}
  where $\tilde a(\cdot,\cdot)$ is equivalent to $\norm{\cdot}{H^{m-k}((0,1))}$.
  We sketch the rather technical extension into 2d and general polynomial degrees in Remark~\ref{rem:ext2d}.
  We assume for our analysis to work a solution decomposition of $u$.
  \begin{assumption}\label{ass:HOdecomp1}
    We assume a decomposition $u=v+w$ into a smooth part $v$ and boundary layer parts $w_1$, $w_2$, for which holds
    \begin{align*}
      |\partial_x^i v(x,y)| &\lesssim 1,&
      |\partial_x^i w_1(x,y)| &\lesssim \eps^{m-k-i}\exp{-x/\eps},
    \end{align*}    
    where $0\leq i\leq 2m$ and analogously for $w_2$.
  \end{assumption}
  The mesh for the problem of this section is a 1d-version of the S-type mesh from the previous section with
  $\Omega_c=(\lambda_x,1-\lambda_x)$ and $\Omega_c^*=(\lambda_x-h,1-\lambda_x+h)$.
  
  The discrete space $V^N$ is the $H^m_0$-conforming space of Hermite-polynomials of degree $p=2m-1$.
  Beside the canonical Hermite-interpolation $I$ we introduce a Ritz-projection $\pi$ into $V^N(\Omega_c)$ by
  \begin{align*}
    \tilde a(v-\pi v,\chi)&=0 \quad \text{in $\Omega_c$ for all }\chi\in V^N(\Omega_c),\\
    \partial_x^n(v-\pi v)&=0 \quad \text{on $\partial\Omega_c$ for all }n\in\{0,\dots,m-k-1\}.
  \end{align*}
  It is well known \cite{Natterer77}, that on the uniform mesh $\Omega_c$ the error bound 
  \begin{gather}\label{eq:Ritz}
    \norm{v-\pi v}{L^\infty(\Omega_c)}\lesssim N^{-(p+1)}
  \end{gather}
  holds for polynomial degrees $p\geq 2$.

%
  Now the second interpolation operator $Pu\in V^N$ is given for $u=v+w$ by
  \begin{align*}
    Pw\big|_\tau&=\begin{cases}
                    Iw\big|_\tau & \tau\subset\Omega\setminus\Omega_c^*,\\
                    0            & \tau\subset\Omega_c,
                  \end{cases}&
    Pv\big|_\tau&=\begin{cases}
                    Iw\big|_\tau    & \tau\subset\Omega\setminus\Omega_c^*,\\
                    \pi w\big|_\tau & \tau\subset\Omega_c.
                  \end{cases}
  \end{align*}
  Note that the definition of $P$ is complete by $Pu\in V^N$. Before we start with the analysis 
  we state a third assumption.
  \begin{assumption}\label{ass:tildea1}
    We assume for the bilinear form $\tilde a(\cdot,\cdot)$ to hold
    \[
      \tilde a(u,v)_{\Omega_c}\lesssim \norm{u}{W^{p,m-k}(\Omega_c)}\norm{v}{W^{q,m-k}(\Omega_c)}
    \]
    for $p=q=2$ and $p=\infty,\,q=1$.
  \end{assumption}
  This assumption is fulfilled for symmetric bilinear forms $\tilde a(\cdot,\cdot)$ equivalent to the $H^{m-k}$-norm.
  
  The analysis can now be conducted as in the previous section. We denote the error components by
  \[
    u-u^N=(u-Pu)+(Pu-u^N)=:\eta+\xi.
  \]

  \begin{lemma}\label{lem:tildea1}
    Let $\sigma\geq 2m=p+1$. Under the Assumption~\ref{ass:HOdecomp1} we have
    \[
      |\tilde a(\eta,\xi)|\lesssim\eps^{1/2}\Big(N^{-(m+k)}\left((Nh)^{k-1}+(\ln N)^{1/2}\right)+\left(h+N^{-1}\max|\psi'|\right)^{m+k}\Big)\enorm{\xi}.
    \]
  \end{lemma}
  \begin{proof}
    The proof follows the proof of Lemma~\ref{lem:reaction} but has some differences in the details. 
    Therefore, we give the full proof here.
    
    We will prove the estimate in the coarse and remaining region separately. 
    Let us start on $\Omega_c$. 
    By definition of $P$ and the orthogonality of the Ritz-error we have
    \[
      |\tilde a(\eta,\xi)_{\Omega_c}|
        =|\tilde a(w,\xi)_{\Omega_c}|
        \lesssim\norm{w}{H^{m-k}(\Omega_c)}\enorm{\xi}
        \lesssim\eps^{1/2} N^{-\sigma}\enorm{\xi}.
    \]
    In the remaining domain we have
    \[
      \tilde a(\eta,\xi)_{\Omega\setminus\Omega_c}
        =\tilde a(u-Iu,\xi)_{\Omega\setminus\Omega_c}+\tilde a(Iu-Pu,\xi)_{\Omega_c^*\setminus\Omega_c}.
    \]
    For the first term it holds with Assumption~\ref{ass:tildea1}
    \begin{align*}
      |\tilde a(u-Iu,\xi)_{\Omega\setminus\Omega_c}|
        &\lesssim\left(\meas^{1/2}(\Omega\setminus\Omega_c)\norm{v-Iv}{W^{m-k,\infty}(\Omega\setminus\Omega_c)}
                +\norm{w-Iw}{H^{m-k}(\Omega\setminus\Omega_c)}\right)\enorm{\xi}\\
        &\lesssim\eps^{1/2}\left(N^{-(m+k)}(\ln N)^{1/2}+\left(h+N^{-1}\max|\psi'|\right)^{m+k}\right)\enorm{\xi},
    \end{align*}
    where the interpolation errors were estimated 
    in the usual way. 
    Local (anisotropic) interpolation error formulas can be found in \cite{SW13}.
    
    For the second term let us look at $\tau=(\lambda-h,\lambda)\subset\Omega_c^*\setminus\Omega_c$, the other interval follows analogously.
    We denote by $\phi_n$ the basis-functions that have as degrees of freedom the $C^n$-compatibility
    at $x=\lambda$ for $n\in\{0,\dots,m-1\}$. Then it holds
    \[
      \norm{\phi_n}{W^{m-k,\infty}(\tau)}\lesssim h^n\left(1+h^{-(m-k)}\right),\quad n\in\{0,\dots,m-1\}.
    \]
    Now we have for the boundary layers $w=w_1+w_2$ and the smooth part $v$
    \[
      Iw-Pw =\sum_{n=0}^{m-1}\partial_x^n w(\lambda)\phi_n\quad\text{ and }\quad
      Iv-Pv =\sum_{n=m-k}^{m-1}\partial_x^n \pi(Iv-v)(\lambda)\phi_n,
    \]
    where the definition of $I$ and the boundary conditions of the Ritz-projection 
    were used in the representations.
    Thus, it follows
    \begin{align*}
      |\tilde a(Iu-Pu,\xi)_{\tau}|
        &\lesssim( \norm{Iv-Pv}{H^{m-k}(\tau)}
                  +\norm{Iw-Pw}{H^{m-k}(\tau)})\enorm{\xi}\\
        &\lesssim\eps^{1/2}( \norm{Iv-Pv}{W^{m-k,\infty}(\tau)}
                  +\norm{Iw-Pw}{W^{m-k,\infty}(\tau)})\enorm{\xi}.
    \end{align*}
    For the first norm we use inverse inequalities and the $L^\infty$-error estimate \eqref{eq:Ritz} of the Ritz-projection
    to obtain
    \begin{align}
      \norm{Iv-Pv}{W^{m-k,\infty}(\tau)}
        &\lesssim \sum_{n=m-k}^{m-1}|\partial_x^n(Iv-\pi v)(\lambda)|\norm{\chi_n}{W^{m-k,\infty}(\tau)}\notag\\
        &\lesssim \sum_{n=m-k}^{m-1}N^n(\norm{Iv-v}{L^\infty(\Omega_c)}+\norm{v-\pi v}{L^\infty(\Omega_c)})h^n\left(1+h^{-(m-k)}\right)\notag\\
        &\lesssim \sum_{n=m-k}^{m-1}N^nN^{-2m}h^n\left(1+h^{-(m-k)}\right)\notag\\
        &\lesssim N^{-(m+k)}\left(1+(Nh)^{k-1}\right),\label{eq:piv-error}
    \end{align}
    while for the second norm we use $h\lesssim \eps$ and $\frac{\eps}{h}\lesssim N$, see \eqref{eq:hbound1}, to obtain
    \begin{align}
      \norm{Iw-Pw}{W^{m-k,\infty}(\tau)}
        &\lesssim \sum_{n=0}^{m-1}|\partial_x^n w(\lambda)|\norm{\chi_n}{W^{m-k,\infty}(\tau)}\notag\\
        &\lesssim \sum_{n=0}^{m-1}\eps^{m-k-n}N^{-\sigma}h^n\left(1+h^{-(m-k)}\right)\notag\\
        &\lesssim N^{-\sigma}\left(1+\left(\frac{\eps}{h}\right)^{m-k}\right)
         \lesssim N^{-(\sigma-(m-k))}.\label{eq:Pw-error}
    \end{align}
    Choosing $\sigma\geq 2m=p+1$ the proof is done by collecting the separate bounds.
  \end{proof}

  \begin{lemma}\label{lem:Hinterpol1}
    Let $\sigma\geq 2m$. Under the Assumptions~\ref{ass:HOdecomp1} and \ref{ass:convex} we have
    \[
      \bnorm{\eta}\lesssim N^{-m}(1+(Nh)^{k-1})+\left(h+N^{-1}\max|\psi'|\right)^{m}.
    \]
  \end{lemma}
  \begin{proof}
    We can follow the proof of Lemma~\ref{lem:interpol} line by line.
  \end{proof}
  
  Combining the results of these lemmas gives the main result for the higher-order case.
  \begin{theorem}\label{thm:highorder1}
    Let $\sigma\geq 2m=p+1$ and Assumptions~\ref{ass:HOdecomp1} and \ref{ass:convex} hold. 
    Then we have for the solutions $u$ of \eqref{3.1} and $u^N$ of the corresponding Galerkin method
    \[
      \bnorm{u-u^N}\lesssim N^{-m}(1+(hN)^{k-1})+\left(h+N^{-1}\max|\psi'|\right)^m.
    \]
  \end{theorem}

  \begin{remark}
    Under the additional assumption $Nh\lesssim 1$, which is equivalent to $h\lesssim N^{-1}$, Theorem~\ref{thm:highorder1}
    yields the shorter estimate
    \[
      \bnorm{u-u^N}\lesssim \left(N^{-1}\max|\psi'|\right)^m=\left(N^{-1}\max|\psi'|\right)^{p+1-m}.
    \]
    This assumption on $h$ is true for the Shishkin mesh with
    \[
      \bnorm{u-u^N}\lesssim (N^{-1}\ln N)^{p+1-m}
    \]
    or the Bakhvalov-S-mesh for $\eps\lesssim N^{-1}$ with
    \[
      \bnorm{u-u^N}\lesssim N^{-(p+1-m)}.
    \]
  \end{remark}
  
  \begin{remark}\label{rem:ext2d}
    For the 2d-case similar ideas can be used. Altogether it is a quite technical but 
    straightforward task. We will show the idea for the case $m=2$ and $k=1$, thus a 
    fourth order-problem with $\tilde a(\cdot,\cdot)$ a second order bilinear form like 
    the one considered in \cite{Xenophontos17} in 1d.
    
    We start with an assumption on a decomposition of $u=v+\sum\limits_{i=1}^4(w_i+c_i)$ 
    into a smooth part $v$, four boundary layer parts $w_i$ with $|\partial_x^i\partial_y^j w_1(x,y)|\lesssim \eps^{1-i}\e^{-x/\eps}$
    and four corner layer parts with $|\partial_x^i\partial_y^j c_1(x,y)|\lesssim \eps^{1-i-j}\e^{-x/\eps}\e^{-y/\eps}$ 
    (and analogously for the remaining parts) for $0\leq i,j\leq 2m$. The mesh is defined as in Section~\ref{sec:2}, our discrete space $V^N$
    is the space of bicubic $C^1$-Hermite-splines and $I$ the canonical Hermite-interpolation into $V^N$.
    
    The main task is to define the projection $P$ into $V^N$. We define it separately for each part of the decomposition.
    Let $\Omega_1:=(0,\lambda)\times(0,1)$ and $\Omega_1^*:=(\lambda-h,\lambda)\times(0,1)$. Then
    \[
      Pw_1|_\tau:=\begin{cases}
                    Iw_1,& \tau\subset\Omega_1\setminus\Omega_1^*,\\
                    0, & \tau\subset\Omega\setminus\Omega_1.
                  \end{cases}
    \]
    Again $Pw_1$ is completely defined by $Pw_1\in V^N$. For the corner-component $c_1$ we define similarly 
    $\widehat\Omega_1:=(0,\lambda)\times(0,\lambda)$, $\widehat\Omega_1^*:=(\lambda-h,\lambda)\times(0,\lambda)\cup(0,\lambda)\times(\lambda-h,\lambda)$
    and
    \[
      Pc_1|_\tau:=\begin{cases}
                    Ic_1,& \tau\subset\widehat\Omega_1\setminus\widehat\Omega_1^*,\\
                    0,   & \tau\subset\Omega\setminus\widehat\Omega_1.
                  \end{cases}
    \]
    For the other layer components we proceed similarly. That leaves the smooth part. 
    With $\Omega_c$ and $\Omega_c^*$ from Section~\ref{sec:2} we define
    \[
      Pv|_\tau:=\begin{cases}
                    Iv,    & \tau\subset\Omega\setminus\Omega_c^*,\\
                    \pi v, & \tau\subset\Omega_c,
                  \end{cases}
    \]
    where $\pi v$ is the Ritz-projection into $V^N_c:=\{v\in C^1(\Omega_c)\,:\,v|_\tau\subset \QS_3(\tau)\}$ given by
    \begin{align*}
      \tilde a(v-\pi v,\chi)_{\Omega_c} &= 0\text{ for all }\chi\in V^N_c\cap H^1_0(\Omega_c)\\
      Iv-\pi v&=0\text{ on }\partial\Omega_c.
    \end{align*}
    Note that the boundary condition implies
    \[
      \partial_t(Iv-\pi v)=0\text{ on }\partial\Omega_c,
    \]
    where $\partial_t$ denotes the tangential derivative.
    
    Given this interpolation operator $P$ it is straightforward to show
    \begin{align*}
      \norm{Iw_i-Pw_i}{W^{1,\infty}(\Omega_i^*)}+
      \norm{Ic_i-Pc_i}{W^{1,\infty}(\widehat\Omega_i^*)}&\lesssim N^{-(\sigma-1)},\\
      \norm{Iv-Pv}{W^{1,\infty}(\Omega_c^*\setminus\Omega_c)}&\lesssim(1+hN)N^{-3},
    \end{align*}
    where the additional assumption on the minimal mesh width $h_{min}\geq\eps N^{-1}$ is needed 
    for the first and an $L^\infty$-error estimation for 
    the Ritz projection or an $L^\infty$-stability result for $\pi$ is assumed for the second estimate 
    (for a fourth-order problem discretised on a triangular mesh by Clough-Tocher elements see \cite{Rannacher76}).
    
    Similarly we obtain
    \begin{align*}
      |\tilde a(\eta,\xi)|&\lesssim\eps^{1/2}(1+hN)\left(h+N^{-1}\max|\psi'|\right)^3\enorm{\xi},\\
      \bnorm{\eta} &\lesssim(1+hN)\left(h+N^{-1}\max|\psi'|\right)^2
    \end{align*}
    for $\sigma\geq 4$ by a tedious estimation. Combining above steps gives the result in the 2d-case
    \[
      \bnorm{u-u^N}\lesssim(1+hN)\left(h+N^{-1}\max|\psi'|\right)^2
    \]
    for $\sigma\geq 4$. Note that $p+1-m=2$ is the convergence order.
    
    The extension of these ideas to the general case of $m\geq k\geq1$ is also clear.
    With $p=2m-1$ we use an $C^{m-1}$-Hermite-space with piecewise $\QS_p$-polynomials and define the
    projection $\pi$ as Ritz-projection using higher order boundary conditions depending on $m-k$. 
    Then for $\sigma\geq 2m=p+1$ the result from Theorem~\ref{thm:highorder1} holds also in 2d.
    
    The final extension of above analysis is to increase the polynomial degree to $p\geq 2m$ while preserving
    the $C^{m-1}$-continuity of the discrete space. With a suitable defined operator $I$
    and a properly defined interpolation operator $P$ (using above $\pi$ and ideas from Section~\ref{sec:2})
    the balanced norm estimate can be shown for $\sigma\geq p+1$ and $hN\lesssim 1$ to be
    \[
      \bnorm{u-u^N}\lesssim\left(N^{-1}\max|\psi'|\right)^{p+1-m}.
    \]
  \end{remark}

\bibliographystyle{plain}
\bibliography{lit}

\end{document}